\documentclass[11pt,twoside,a4paper]{article}
\usepackage{amsmath}
\usepackage{amssymb}

\usepackage{eufrak}
\usepackage{euscript}

\newtheorem{theorem}{Theorem}[section]
\newtheorem{lemma}[theorem]{Lemma}
\newtheorem{proposition}[theorem]{Proposition}

\newcommand{\vb}{\vspace{3mm}}

\newcommand{\DROU}{\mbox{\sc drou}}
\newcommand{\OU}{\mbox{\sc ou}}
\newcommand{\ROU}{\mbox{\sc rou}}
\newcommand{\SDE}{\mbox{\sc sde}}
\newcommand{\CLT}{\mbox{\sc clt}}

\newcommand{\DD}{{\rm d}}
\newcommand{\s}{^\star}

\newcommand{\rr}{\mathrm{R}}
\newcommand{\one}{\mathbf{1}}
\newcommand{\pp}{\mathbb{P}}

\newenvironment{proof}[1][\proofname]{\par \normalfont \trivlist
 \item[\hskip\labelsep\itshape #1]\ignorespaces
}{%
 \hspace*{\fill}$\Box$ \endtrivlist
}
\newcommand{\proofname}{{\bf Proof}}

\begin{document}
\title{{Limit theorems for reflected \\Ornstein-Uhlenbeck processes}}

\author{Gang Huang, Michel Mandjes \& Peter Spreij}
\maketitle
\begin{abstract} \noindent
This paper studies one-dimensional Ornstein-Uhlenbeck processes, with the distinguishing feature that they are reflected on a single boundary (put at level $0$) or two boundaries (put at levels $0$ and $d>0$).
In the literature they are referred to as reflected $\OU$ ($\ROU$) and doubly-reflected $\OU$ ($\DROU$) respectively. For both cases, we explicitly determine the decay rates of the (transient) probability to reach a given extreme level. The methodology relies on sample-path large deviations, so that we also identify the associated most likely paths.
For $\DROU$, we also consider the `idleness process' $L_t$ and the `loss process' $U_t$, which are the minimal nondecreasing processes
which make the $\OU$ process remain $\geqslant 0$ and $\leqslant d$, respectively. We derive central limit theorems ($\CLT$\,s) for $U_t$ and $L_t$, using techniques from stochastic integration and the martingale $\CLT$.

\vb


\vb

\noindent {\it Affiliations.}
The authors are with Korteweg-de Vries Institute for Mathematics, University of Amsterdam, Science Park 904, 1098 XH  Amsterdam,
the Netherlands; the second author is also with CWI, Amsterdam, the Netherlands, and Eurandom, Eindhoven University of Technology, the Netherlands.

\vb

\noindent {\it Email}. $\{$\tt{g.huang|m.r.h.mandjes|p.j.c.spreij}$\}$\tt{@uva.nl}.
\end{abstract}

\newpage
\section{Introduction}
Ornstein-Uhlenbeck ($\OU$) processes are Markovian, mean reverting Gaussian processes. They well describe various real-life phenomena, and allow a relatively high degree of analytical tractability. As a result, they have found wide-spread use  in a broad range of application domains, such as finance, life sciences, and operations research. In many situations, though, the stochastic process involved is not allow to cross a certain boundary, or is even supposed to remain within two boundaries. The resulting reflected (denoted in the sequel by $\ROU$)
and doubly-reflected ($\DROU$) $\OU$ processes have hardly been studied, though, a notable exception being the works by Ward and Glynn \cite{MR1957808,MR1993278,MR2172907}, where  $\ROU$ processes 
are used to approximate the number-in-system processes in M/M/1 and GI/GI/1 queues with reneging under a specific, reasonable scaling; the $\DROU$ process can be seen as an approximation of the associated finite-buffer queue. 
Srikant and Whitt \cite{Srikant:1996:SRL:229493.229496} also show that the number-in-system process in a GI/M/$n$ loss model can be approximated by $\ROU$. For other applications, we refer to e.g.\ the introduction of \cite{MR2944002} and references therein.

\vb

As known, the $\OU$ process is defined
as the unique strong solution to the stochastic differential equation ($\SDE$):
\[
\DD X_{t}=(\alpha-\gamma X_{t})\DD t+\sigma \DD B_{t} , \:\:\:\:X_0=x \in \mathrm{R},
\]
where $\alpha\in\rr$, $\gamma, \sigma >0$ and $B_{t}$ is a standard Brownian motion. This process is {\it mean-reverting} towards the value $\alpha/\gamma$. To incorporate reflection at a lower boundary $0$, thus constructing $\ROU$, the following $\SDE$ is used, where we throughout the paper additionally assume $\alpha>0$,
\[\DD Y_{t}=(\alpha-\gamma Y_{t})\DD t+\sigma \DD B_{t}+\DD L_{t}, \:\:\:\:Y_{0}=x\geqslant 0,
\]
where $L_{t}$ could be interpreted as an `idleness process'. More precisely, $L_t$ is defined as the minimal nondecreasing process such that $Y_{t}\geqslant 0$ for $t\geqslant 0$;  it holds that $\int_{[0, T]}\one_{\{Y_{t}>0\}}\DD L_{t}=0$ for any $T>0$. 

Likewise, reflection at two boundaries can be constructed. $\DROU$ is defined through
the $\SDE$
\[ \DD Z_{t}=(\alpha-\gamma Z_{t})\DD t+\sigma \DD B_{t}+\DD L_{t}-\DD U_{t}, \:\:\:\:Z_{0}=x \in [0,d],
\]
where $U_{t}$ is the `loss process' at the boundary $d$, i.e., we have
$\int_{[0, T]}\one_{\{Z_{t}>0\}}\DD L_{t}=0$ as well as
$\int_{[0, T]}\one_{\{Z_{t}< d\}}\DD U_{t}=0$ for any $T>0$. In the case of $\DROU$ we assume that the upper boundary $d$ is larger than ${\alpha}/{\gamma}$ throughout this paper, to guarantee that hitting $d$ does not happen too frequently (which is a reasonable assumption for most of applications). For the existence of a unique solution to the above $\SDE$s with reflecting boundaries, we refer to e.g.\ \cite{MR529332}. In the context of queues with finite-capacity, $U_{t}$ is the continuous analog to the cumulative amount of loss over $[0, t]$, and that explains why we refer to it as the `loss process'.  

\vb

A first objective of this paper is to obtain insight into transient rare-event probabilities. We do so
for an $\ROU$ process with `small perturbations', that is, a process given through the $\SDE$
\[
\DD Y^{\epsilon}_{t}=(\alpha-\gamma Y^{\epsilon}_{t})\DD t+\sqrt{\epsilon}\sigma \DD B_{t}+\DD L^{\epsilon}_{t},\]
with $\epsilon>0$ typically small. The transient distribution (at time $T\geqslant0$, for any initial value $x\geqslant0$) of the $\OU$ process being explicitly known (it actually has a Normal distribution), we lack such results for the $\ROU$ process. (As an aside, we note that the {\it stationary} distribution of $\ROU$ {\it is} known \cite{MR1993278}; it is a truncated Normal distribution.)
The above motivates the interest in large-deviations asymptotics of the type
\begin{equation}\label{DRA}
\lim_{\epsilon\rightarrow 0}\epsilon \log\mathbb{P}(Y^{\epsilon}_{T}\geqslant b\mid Y^{\epsilon}_{0}=x),
\end{equation}
for $x\geqslant0$, $T\geqslant0$, and $b>{\mathbb E} (Y^{\epsilon}_{T}\mid Y^{\epsilon}_{0}=x)$ (so that the event under consideration is rare). We follow the method used for computing blocking probabilities of the Erlang queue in \cite{MR1335456}, that is, relying on sample-path large deviations. In our strategy, a first step is to study the above decay rate for the `normal' (that is, non-reflected) $\OU$ process. This decay rate is computed as the solution of a certain variational problem, relying on standard calculus-of-variations: it minimizes an `action functional' over all paths $f$ such that $f(0)=x$ and $f(T)\geqslant b$. The optimizing path $f\s$ has the informal interpretation of `most likely path' (or: `minimal cost path'): given the rare event under study happens, with overwhelming probability it does so through a path `close to' $f\s$.
The next step is to observe that $f\s$ does not hit level $0$ between $0$ and $T$, and hence the decay rate that we found also applies for $\ROU$ (rather than $\OU$).

The computations for $\OU$ are presented in Section \ref{SEC2}. The results are in line with what could be computed from the explicitly known distribution of $X^{\epsilon}_{T}$ conditional on $X^{\epsilon}_0=x$, but provide us in addition with the most likely path. 
Section \ref{SEC3} then focuses on the computation of the decay rate for $\ROU$. Above we described the intuitively appealing approach we followed, but it should be emphasized that at the technical level there are some non-trivial steps to be taken. 
The primary complication is that the local large-deviations rate function at the reflecting boundary is
different from this function in the interior \cite{MR770425}. Inspired by \cite{MR0397893}, 
we derive explicit expressions of the large-deviations rate function for $\ROU$ by properties of the reflection map in the deterministic Skorokhod problem. 
Unfortunately, calculus-of-variation techniques cannot be used immediately to identify 
the most likely path; this is due to the fact that we need to minimize over all non-negative continuous paths. However, the non-negativity of the optimizing path for the $\OU$ process
facilitates the computation of the decay rates for $\ROU$. In Section 4, we compute the decay rate for $\DROU$ by the same strategy as the one for $\ROU$.

\vb

The second part of the paper focuses on $\DROU$, with emphasis on properties of the loss process $U_t$ (and also the idleness process $L_t$), for $t$ large. Zhang and Glynn's martingale approach, as developed in \cite{MR2771195}, is employed to tackle a problem of this type. With $h(\cdot)$ being a twice continuously differentiable real function, we apply It\^{o}'s formula on $h(Z_{t})$ and require $h(\cdot)$ to satisfy certain ordinary differential equations ({\sc ode}\,s) and specific initial and boundary conditions in order to construct martingales related to $U_{t}$ and $L_{t}$. The presence of $Z_{t}$ in the drift term leads to {\sc ode}\,s with nonconstant coefficients, which seriously complicates the derivation of exact solutions. In Section \ref{SEC5} we use this approach to identify a $\CLT$ for $U_t$: we find explicit expressions for $q_U$ and $\eta_U$ such that $(U_t-q_U t)/\sqrt{t}$ converges to a Normal random variable with mean $0$ and variance $\eta_U^2$; a similar result is established for $L_t$. In Section \ref{SEC6}, we discuss 
the corresponding large deviations probability, that is the probability that $U_{t}/t$ exceeds a given threshold $c$ larger than $Q_U$, for $t$ large.

\section{Transient asymptotics for Ornstein-Uhlenbeck}\label{SEC2}
The primary goal of this section is to compute the decay rate (\ref{DRA}) with $Y^\epsilon$ replaced by $X^\epsilon$; in other words, we now consider the $\OU$ case (that is, no reflection). Before we attack this problem, we first identify the $\OU$ process' average behavior. To this end, we first describe the so-called `zeroth-order approximation' of one-dimensional diffusion processes. The $\SDE$ (more general than the one defining $\OU$) we here consider is
\[
\DD J^{\epsilon}_{t}=b(J^{\epsilon}_{t})\DD t+\sqrt{\epsilon} \sigma(J^{\epsilon}_{t})\DD B_{t}, \:\:\:\:J^{\epsilon}_{0}=x,
\]
and the corresponding {\sc ode} is 
\begin{equation*}\label{eq:ouode}
\DD x(t)=b(x(t))\DD t,\,\,x(0)=x.
\end{equation*}

\begin{theorem} {\rm \cite[Thm.\ 2.1.2]{MR722136}}
Suppose that $b(\cdot)$ and $\sigma(\cdot)$ are Lipschitz continuous and increase no faster than linearly, i.e.,
\begin{equation*}
[b(x)-b(y)]^{2}+[\sigma(x)-\sigma(y)]^{2}\leqslant K^{2}|x-y|^{2},
\end{equation*}
\begin{equation*}
b^{2}(x)+\sigma^{2}(x)\leqslant K^{2}(1+|x|^{2}),
\end{equation*}
where $K$ is a constant. Then for all $t>0$ and $\epsilon>0$ we have
\begin{equation*}
{\mathbb E}|J^{\epsilon}_{t}-x(t)|^{2}\leqslant \epsilon a(t),
\end{equation*}
where $a(t)$ is a monotone increasing function, which is expressed in terms of $|x|$ and $K$.
Moreover, for all $t>0$ and $\delta>0$
\begin{equation*}
\lim_{\epsilon \rightarrow 0}{\mathbb P}\left(\sup_{0\leqslant s\leqslant t}|J^{\epsilon}_{s}-x(s)|>\delta\right)=0.
\end{equation*}
\end{theorem}
In the specific case of $\OU$ processes, the corresponding small perturbation process $X^{\epsilon}_{t}$ (on a finite time interval) satisfies
\begin{equation}\label{eq:oueps}
\DD X^{\epsilon}_{t}=(\alpha-\gamma X^{\epsilon}_{t})\DD t+\sqrt{\epsilon}\sigma \DD B_{t},
 \:\:\:\:X_0=x \geqslant 0.
\end{equation}
 It is readily checked that the limiting process $x(t)$ is given by 
\[
\dot{x}(t)= \alpha-\gamma x(t),\,\, x(0)=x, 
\]
which has the solution
 \[x(t) = \frac{\alpha}{\gamma}+\left(x-\frac{\alpha}{\gamma}\right)e^{-\gamma t}.\] 
Note that $x(t)=\mathbb{E}X^\epsilon_t$. Popularly,
as $\epsilon\downarrow 0$, with high probability $X^{\epsilon}_{t}$ is contained  in any $\delta$-neighborhood of $x(t)$ on the interval $[0, T]$. Assuming that $b> x(T)$, it is now seen that the probability of our interest, of which we wish to identify the decay rate, relates to a rare event.

\vb

We now recall the Freidlin-Wentzell theorem \cite[Thm.\ 5.6.7]{MR1619036}, which is the cornerstone behind the results of this section.
To this end, we first define $C_{[0,T]}({\mathbb R})$ as the space of continuous functions from $[0,T]$ to ${\mathbb R}$, with the uniform norm $\|f\|_{\infty}:=\sup_{t\in[0, T]} |f(t)|$ and the metric $d(f,g):=\|f-g\|_{\infty}$. The Freidlin-Wentzell result now states that $X^\epsilon$ satisfies the sample-path large deviations principle ({\sc ldp}) with the good rate function
\[
I_{x}(f):=\begin{cases}
(2\sigma^{2})^{-1}\int_{0}^{T}(f'(t)-\alpha+\gamma f(t))^{2}\DD t & \text{if } f\in H_x,\\
\infty & \text{if } f \notin H_x,
\end{cases}\]
where $H_x:=\{f: f(t)=x+\int_{0}^{t}\phi(s)\DD s, \phi \in L_{2}([0,T])\}$. The {\sc ldp} states that
for any closed set $F$ and open set $G$ in $(C_{[0,T]}({\mathbb R}), \|\cdot\|_{\infty})$,
\begin{eqnarray*}
\limsup_{\epsilon \rightarrow 0} \epsilon \log \mathbb{P}(X^{\epsilon}_{t}\in F)&\leqslant& -\inf_{f \in F} I_{x}(f),\\
\liminf_{\epsilon \rightarrow 0} \epsilon \log \mathbb{P}(X^{\epsilon}_{t}\in G)&\geqslant& -\inf_{f \in G} I_{x}(f).
\end{eqnarray*}
These upper and lower bounds obviously match for $I_x$-continuity sets $S$, that is, sets $S$ such that $\inf_{f\in \mathrm{cl}\,S}I_{x}(f)=\inf_{f\in \mathrm{int}\,S}I_{x}(f)$.

\vb

We now return to the decay rate under consideration. Let us first introduce some notation, following standard conventions in Markov process theory. We write $\pp_x(E)$ for the probability of an event $E$ in terms of the process $X^\epsilon$ if this process starts in $x$. We will mainly work with a fixed time horizon $T>0$ and write $X_\centerdot$ for $\{X_t,\, t\in [0,T]\}$.
Our first step is to express the probability under study in terms of probabilities featuring in the sample-path {\sc ldp}. Observe that
we can write $\pp_x(X^\epsilon_T\geqslant b) = {\mathbb P}_x(X^\epsilon_\centerdot\in S),$
with
\[S:=\bigcup_{a\geqslant b} S_a, \:\:\:S_a
:=\left\{f\in C_{[0,T]}({\mathbb R}): f(0)=x, f(T)=a\right\}.\]
Below we first solve a calculus-of-variation problem to find $\inf_{f \in S_a} I_{x}(f)$ explicitly. Secondly, we prove that $S$ is an $I_{x}$-continuity set. A combination of these findings gives us an expression for the decay rate.

\begin{proposition}\label{PROP1} Let $a\geqslant b> x(T)$.
Then
\[
\inf_{f \in S_{a}}I_{x}(f)=
\frac{[a-x(T)]^{2}}{[1-e^{-2\gamma T}] (\sigma^{2}/\gamma)}=\frac{[a-(\frac{\alpha}{\gamma}+(x-\frac{\alpha}{\gamma})e^{-\gamma T})]^{2}}{[1-e^{-2\gamma T}] (\sigma^{2}/\gamma)}.
\]
The optimizing path is given by
\[
f\s(t)=(C-\frac{\alpha}{\gamma})e^{\gamma t}+(x-C)e^{-\gamma t} +\frac{\alpha}{\gamma},
\:\:\:\mbox{where}\:\:\:
C:=\frac{a-\frac{\alpha}{\gamma}+\frac{\alpha}{\gamma}e^{\gamma T}-x e^{-\gamma T}}{e^{\gamma T}-e^{-\gamma T}}.\]
Moreover, $f\s(t)\geqslant 0$ on $t \in [0, \infty)$ when the starting point $x\geqslant 0$; $f\s(t)\in [0, d]$ on $t \in [0, T]$ when the starting point $x\in [0, d]$, $a\in [0, d]$ and ${\alpha}/{\gamma}<d$.
\end{proposition}

\begin{proof}
Obviously,
\[
\inf_{f \in S_{a}} I_{x}(f)
=
 \inf \left\{\frac{1}{2\sigma^{2}}\int_{0}^{T}(f'(t)+\gamma f(t)-\alpha)^{2}\DD t, \:f\in H_x\cap S_{a}\right\}.\]
According to  Euler's necessary condition \cite[Thm.\ C.13]{MR1335456}, the initial condition and the boundary condition, we have that the optimizing path satisfies 
\begin{equation*}
f''(t)-\gamma^2 f(t)+\alpha\gamma=0, \:\:\:\:f(0)=x, \:\:\:\:f(T)=a.
\end{equation*}
The general solution of the {\sc ode} (unique up to the choice of the two constants) reads
\begin{equation*}
f(t)=C_{1} e^{\gamma t}+C_{2}e^{-\gamma t}+\frac{\alpha}{\gamma}.
\end{equation*}
It is now readily checked that the stated expression follows, by imposing the initial condition and the boundary condition.
Hence,
\begin{equation*}
\inf_{f \in S_{a}}I_{x}(f)
= \frac{(2C_1\gamma)^{2}}{2\sigma^{2}}\int_{0}^{T}e^{2\gamma t}\DD t =\frac{[a-(\frac{\alpha}{\gamma}+(x-\frac{\alpha}{\gamma})e^{-\gamma T})]^{2}}{[1-e^{-2\gamma T}] (\sigma^{2}/\gamma)}.
\end{equation*}
We proceed with proving that $f\s(t)\geqslant 0$ and $f\s(t)\in [0,d]$ on $t \in [0, \infty)$ under the two stipulated assumptions. 
First we
note that $x(t)=xe^{-\gamma t}+(1-e^{-\gamma t})\,{\alpha}/{\gamma}$, a convex combination of $x$ and ${\alpha}/{\gamma}$. Since both of these are nonnegative by assumption, so is $x(t)$.
For $f\s(t)$ we have the following alternative expressions with $q(t):={\sinh(\gamma t)}/{\sinh (\gamma T)}$, as a direct computation shows:
\begin{align*}
f\s(t) & = x(t)+(a-x(T))q(t) \\
& = q(t)a+(e^{-\gamma t}-q(t)e^{-\gamma T})x+\big(1-e^{-\gamma t}-q(t)(1-e^{-\gamma T})\big)\frac{\alpha}{\gamma}.
\end{align*}
It follows from the first equality that $f\s(t)\geqslant x(t)$, because $a \geqslant x(T)$, and hence $f\s(t)$ is nonnegative. Moreover, the second equality shows that $f\s(t)$ is a convex combination of $a$, $x$ and ${\alpha}/{\gamma}$, see below. Since all three of these are assumed to be less than $d$, the same holds true for $f\s(t)$. 
%
Finally we show that we indeed have the claimed convex combination, by showing that all coefficients are nonnegative and sum to one. The latter is obvious, as well as $q(t)\in [0,1]$. Furthermore $e^{-\gamma t}-q(t)e^{-\gamma T}\geqslant (1-q(t))e^{-\gamma T}\geqslant 0$. To prove that the third coefficient is nonnegative we use the basic equality
\[
\sinh (x)=\frac{(1+e^x)(1-e^{-x})}{2}.
\]
Then observe that 
\begin{align*}
1-e^{-\gamma t}-q(t)(1-e^{-\gamma T}) & = 1-e^{-\gamma t}-\frac{(1+e^{\gamma t})(1-e^{-\gamma t})}{(1+e^{\gamma T})(1-e^{-\gamma T})}(1-e^{-\gamma T}) \\
& = (1-e^{-\gamma t})\left(1-\frac{1+e^{\gamma t}}{1+e^{\gamma T}}\right)\geqslant 0.
\end{align*}
This completes the proof.
\end{proof}

\vb

\begin{proposition} \label{PROP2} $S$ is an $I_{x}$-continuity set.\end{proposition}
\begin{proof} Consider the topological space $(C_{[0,T]}({\mathbb R}), \tau)$, where the topology $\tau$ is induced by the metric $d(f,g)$.
We next consider $\bar S_{x}=\{f \in C_{[0,T]}({\mathbb R}):f(0)=x\}$ with the subspace topology $\tau_{\bar S_{x}}=\{U \cap \bar S_{x}: U \in \tau\}$. The set $S$ is a closed subset in $\bar S_{x}$ since 
the coordinate mapping $f\mapsto f(T)$ is $\tau$-continuous. By the same property and the fact that the coordinate mapping is $\tau$-open, the $\tau_{\bar S_{x}}$-interior of $S$ is $\mathrm{int}\, S=\{f \in C_{[0,T]}({\mathbb R}):f(0)=x, f(T)> b\}$.
%
We thus have
\begin{equation*}
\inf_{f \in \mathrm{cl} \,S}I_{x}(f)
= \inf_{f \in S}I_{x}(f)
= \inf_{a\geqslant b} \inf_{f \in S_{a}}I_{x}(f),
\:\:\:\mbox{and}\:\:\:
\inf_{f \in \mathrm{int}\, S}I_{x}(f)= \inf_{a> b} \inf_{f \in S_{a}}I_{x}(f).
\end{equation*}
Using Proposition~\ref{PROP1}
and the fact that $a\geqslant b>x(T)$,  \[
\inf_{a\geqslant b} \inf_{f \in S_{a}}I_{x}(f)=\inf_{a> b} \inf_{f \in S_{a}}I_{x}(f)=\frac{[b-(\frac{\alpha}{\gamma}+(x-\frac{\alpha}{\gamma})e^{-\gamma T})]^{2}}{[1-e^{-2\gamma T}] (\sigma^{2}/\gamma)}.\]Consequently, $S$ is an $I_{x}$-continuity set. \end{proof}

\vb

Now the decay rate under consideration can be determined.
\begin{proposition} \label{PROP3} Let $b>x(T)$. Then
\[
\lim_{\epsilon\rightarrow 0}\epsilon \log\mathbb{P}_x(X^{\epsilon}_{T}\geqslant b)=-\frac{[b-(\frac{\alpha}{\gamma}+(x-\frac{\alpha}{\gamma})e^{-\gamma T})]^{2}}{[1-e^{-2\gamma T}] (\sigma^{2}/\gamma)}.\]
Moreover, the minimal cost path is as given in Proposition~\ref{PROP1} (with $a$ replaced by $b$).
\end{proposition}
\begin{proof}
Apply  `Freidlin-Wentzell' to the $I_{x}$-continuity set $S$:
\begin{eqnarray*}
\lim_{\epsilon\rightarrow 0}\epsilon \log\mathbb{P}_x(X^{\epsilon}_{T}\geqslant b)
&=& \lim_{\epsilon\rightarrow 0}\epsilon \log\mathbb{P}_x(X^{\epsilon}_\centerdot\in S)
\\&=&-\inf_{f \in S} I_{x}(f)= -\inf_{a\geqslant b}\inf_{f \in S_{a}}I_{x}(f).
\end{eqnarray*}
By the computations in the proof of Proposition~\ref{PROP2}, we obtain the desired result. The minimal cost path is directly obtained from Proposition~\ref{PROP1}. \end{proof}

\vb

We mentioned in the introduction that there is an alternative method to compute
the decay rate under study. It follows relatively directly from the fact that  $X^{\epsilon}_{T}$ (with $X^{\epsilon}_{0}=x$) is normally distributed with mean
$\mu_{T}=x(T)=\frac{\alpha}{\gamma}(1-e^{-\gamma T})+xe^{-\gamma T}$ and variance $\sigma_{T}^{2}(\epsilon)=\frac{\epsilon \sigma^2}{2\gamma} (1-e^{-2 \gamma T})$, in conjunction with the standard inequality \cite[p.\ 19]{MR1335456}
\begin{equation*}
\frac{1}{y+y^{-1}}e^{-\frac{1}{2}y^{2}}\leqslant\int_{y}^{\infty}e^{-\frac{1}{2}t^{2}}\DD t\leqslant \frac{1}{y}e^{-\frac{1}{2}y^{2}}.
\end{equation*}
We have followed our sample-path approach, though, for two reasons: (i)~the resulting most likely path is interesting in itself, as it gives insight into the behavior of the system conditional on the rare event, but, more importantly, (ii)~it is useful when studying the counterpart of the decay rate for $\ROU$ (rather than $\OU$), which we pursue in Section \ref{SEC3}.

We also note that
\begin{equation*}
\lim_{T \rightarrow \infty} \lim_{\epsilon \rightarrow 0}\epsilon \log\mathbb{P}_x(X^{\epsilon}_{T}\geqslant b)
=-\frac{(b-\frac{\alpha}{\gamma})^{2}}{\sigma^{2}/\gamma}.
\end{equation*}\\
It is known that the steady-state distribution of $X^{\epsilon}_{t}$ with $X^{\epsilon}_{0}=x$ is normally distributed with mean ${\alpha}/{\gamma}$ and variance ${\epsilon \,\sigma^2}/({2\gamma})$. 
We conclude that this shows that the result is invariant under changing the orders of taking limits ($T\to\infty$ and $\epsilon\to 0$).

\section{Transient asymptotics for reflected Ornstein-Uhlenbeck}\label{SEC3}
This section determines the decay rate (\ref{DRA}) for $\ROU$. For the moment we consider a setting more general than $\OU$ and $\ROU$, namely stochastic differential equations with reflecting boundary conditions.
Let $D^{\circ} \in \mathbb{R}$ be an open interval, $\partial D$ and $D$ denote its boundary and closure.
Let $\nu(x)$ denote the function giving the inward normal at $x \in \partial D$, i.e.\ $\nu(x)=1$ if $x$ is a finite left endpoint of $D$ and $\nu(x)=-1$ if $x$ is a finite right endpoint of $D$.
The reflected diffusion $H^{\epsilon}$ w.r.t. $D$ is defined as the unique strong solution to 
\begin{equation*}
\mathrm{d}H_{t}^{\epsilon}= b(H_{t}^{\epsilon})\mathrm{d}t+ \sqrt{\epsilon}\sigma \mathrm{d}B_{t}+\mathrm{d}\xi_{t}^{\epsilon}, \:\:\:\:H_{0}^{\epsilon}=x\in D,
\end{equation*}
where $|\xi^{\epsilon}|_{t}=\int_{0}^{t}\one_{\partial D}(H_{s}^{\epsilon})\mathrm{d}|\xi^{\epsilon}|_{s}$ and $\xi_{t}^{\epsilon}=\int_{0}^{t}\nu(H_{s})\mathrm{d}|\xi^{\epsilon}|_{s}$. Here
$|\xi^{\epsilon}|_{t}$ denotes the total variation of $\xi^{\epsilon}$ by time $t$. We assume that $b(\cdot)$ is uniformly Lipschitz continuous and grows no faster than linearly, 
and $\sigma$ is a nonzero constant. The existence and uniqueness of the strong solution is proved in \cite{MR529332}.

Next we consider the solution to the deterministic Skorokhod problem for $D$: given $\alpha  \in C_{[0, \infty)}(\mathbb{R})$, there exists a unique pair $(h, \beta)$ such that $h \in C_{[0, \infty)}(D)$, and
$\beta \in C_{[0, \infty)}(\mathbb{R})$ of locally bounded variation that satisfy
\begin{equation*}
h_t=\alpha_t+\beta_t, \:\:\:\:|\beta|_{t}=\int_{0}^{t}\one_{\partial D}(h_{s})\mathrm{d}|\beta|_{s}, \:\:\:\:\beta_{t}=\int_{0}^{t}\nu(h_{s})\mathrm{d}|\beta|_{s}
\end{equation*}
We now recall the sample-path {\sc ldp} for the reflected diffusion process, since it is considerably less known than the (standard) Freidlin-Wentzell theorem for the non-reflected case. We denote by $H^+_x$ the nonnegative functions in $H_x$ and by $\omega$ a function from $[0,T]$ to $\rr$.

\begin{theorem}\label{LDPRD} \emph{(Doss and Priouret \cite[Thm.\ 4.2]{MR770425})} If $b(\cdot)$ is uniformly Lipschitz continuous and bounded, and $\sigma$ is a nonzero constant, 
then $H^{\epsilon}$ satisfies the {\sc ldp} in $C_{[0, T]}(D)$ with the rate function 
\begin{equation*}
I(h)=\inf_{\omega \geqslant 0}\frac{1}{2\sigma^{2}} \int_{0}^{T} (h'_{t}-b(h_{t})-\nu(h_{t})\omega_{t}\one_{\partial D}(h_{t}))^{2}\mathrm{d}t.
\end{equation*}
if $h \in H_x^+$ and $\infty$ else.
\end{theorem}
For reflected diffusions with a single reflecting boundary at $0$, we can identify $\omega(t)$ and have the following explicit expression of
the rate function. As usual, we define $x^{+}=\max\{x, 0\}$ and $x^{-}=-\min\{x, 0\}$.
 \begin{proposition}\label{PROP5} Let $D=[0, \infty)$.
When $b(0)\geqslant 0$, $H^{\epsilon}$ satisfies the LDP in $C_{[0,T]}([0, \infty))$ with the rate function 
\[
I^{+}(h)=
\frac{1}{2\sigma^{2}}\int_{0}^{T}\left(h'_{t}-b(h_{t})\right)^{2}\mathrm{d}t
\]
if $h \in H_x^+$ and $\infty$ else.
When $b(0)<0$, $H^{\epsilon}$ satisfies the LDP in $C_{[0,T]}([0, \infty))$ with the rate function 
\[
I^{+}(h)=
\frac{1}{2\sigma^{2}}\int_{0}^{T}\left(h'_{t}-b(h_{t})\right)^{2}\mathrm{d}t-\frac{1}{2\sigma^{2}}b(0)^{2}\int_{0}^{T}\one_{\{0\}}(h_{t})\,\mathrm{d}t.\]
if $h \in H_x^+$ and $\infty$ else. In short, for $h\in H_x^+$ and $b(0)\in{\mathbb R}$ we have
\[
I^{+}(h)=
\frac{1}{2\sigma^{2}}\int_{0}^{T}\left(h'_{t}-b(h_{t})\right)^{2}\mathrm{d}t-\frac{1}{2\sigma^{2}}(b(0)^-)^2\int_{0}^{T}\one_{\{0\}}(h_{t})\,\mathrm{d}t.
\]
\end{proposition}

\begin{proof} In this case, $D=[0, \infty)$, $\partial D=\{0\}$ and $\nu(0)=1$. The rate function becomes
 \begin{equation*}
I^{+}(h)=\inf_{\{\omega_{t}\geqslant 0\}}\frac{1}{2\sigma^{2}} \int_{0}^{T} (h_{t}^{'}-b(h_{t})-\omega_{t}\one_{\{0\}}(h_{t}))^{2}\,\mathrm{d}t.
\end{equation*}
We minimize for each $t$ separately under the integral.
If $h'_{t}-b(h_{t})<0$, then $\omega_{t}=0$ is optimal. 
If $h'_{t}-b(h_{t})\geqslant 0$ and $h_{t}>0$, then $\one_{\{0\}}(h_{t})\omega_{t}\equiv 0$, which means that any value of $\omega$ is optimal.
If $h'_{t}-b(h_{t})\geqslant 0$ and $h_{t}=0$, then $\omega_{t}=h'_{t}-b(h_{t})$ is optimal.
Hence $\omega^{\star}_{t}=(h'_{t}-b(h_{t}))^{+}$ is the optimizer. It gives the following explicit expression:
\begin{equation*}
I^{+}(h)=
\frac{1}{2\sigma^{2}}\int_{0}^{T}\left(h'_{t}-b(h_{t})-\one_{\{0\}}(h_{t})(h'_{t}-b(h_{t}))^{+}\right)^{2}\mathrm{d}t
\end{equation*}
if $h \in H_x^+$ and $\infty$ else. 
For any $h \in C_{[0,T]}([0, \infty))$ which is differentiable a.e., note that $h'_{t}=0$ if $h_{t}=0$. 
Then we have
\begin{eqnarray*}
\lefteqn{I^{+}(h)
= \frac{1}{2\sigma^{2}}\int_{0}^{T}\left(h'_{t}-b(h_{t})-\one_{\{0\}}(h_{t})b(0)^{-}\right)^{2}\mathrm{d}t}\\
& = & \frac{1}{2\sigma^{2}}\int_{0}^{T}\left(h'_{t}-b(h_{t})\right)^{2}\mathrm{d}t+\:\frac{1}{2\sigma^{2}}\int_{0}^{T}\one_{\{0\} }(h_{t})\left(b(0)^{-}\right)^{2}\mathrm{d}t+\frac{1}{\sigma^{2}}\int_{0}^{T}\one_{\{0\}}(h_{t})b(0)^{-}b(0)\,\mathrm{d}t.
\end{eqnarray*}
When $b(0)\geqslant 0$, the last two terms are zero, and for $b(0)<0$ they sum to
\[
-\frac{1}{2\sigma^2}b(0)^2\int_0^T\one_{\{0\}}(h_{t})\,\DD t.
\]
This completes our proof.
\end{proof}

\vb

Theorem~\ref{LDPRD} requires $b(\cdot)$ to be bounded, but this is a condition that $\ROU$ does not satisfy. A careful inspection of the proof 
in \cite{MR770425} or \cite{MR899955}, however, reveals that if $\sigma$ is a constant the boundedness requirement for $b(\cdot)$  can be dropped.
Specifically, for constant $\sigma$, this proof can be modified in the sense that one can directly apply the contraction principle, which needs uniform continuity of $b(\cdot)$ only.
As a result, Proposition~\ref{PROP5} is valid for $\ROU$. Above we observed (i) that the most likely path for $\OU$ was non-negative (Proposition~\ref{PROP1}),
(ii) the rate functions $I$ and $I^+$ for $\OU$ and $\ROU$ are the same as long as their arguments are nonnegative paths on $[0,T]$ (Proposition~\ref{PROP5}).
This suggests that the decay rates for $\OU$ and $\ROU$ (and the corresponding most likely paths) coincide.

The idea is now that we find the decay rate (\ref{DRA}) for $\ROU$ by using the sample-path results that we derived in the previous section for $\OU$.
Recall that the zeroth-order approximation of $\OU$ is $x(t)={\alpha}/{\gamma}+(x-{\alpha}/{\gamma})e^{-\gamma t}$. It is readily checked that $x(t)>0$ when the starting point $x\geqslant 0$. So we still assume $b> x(T)$ in the decay rate \label{DR} for $\ROU$. We define $S^{+}:=\{f \in C_{[0,T]}([0, \infty)): f(0)=x, f(T)\geqslant b\}$, corresponding to the rare event $\{Y^{\epsilon}_{\centerdot}\in S^{+}\}$,  so as to compute the decay rate (\ref{DRA}); the set $S_a^+$ is defined as $\{f \in C_{[0,T]}([0, \infty)):f(0)=x, f(T)=a \}$. Below we keep the notation $\pp_x$ for probabilities of events in terms of $Y^\epsilon$ when this process starts in $x$.

\begin{theorem} \label{PROP6} Let $b> x(T)$. Then, similar to the result of Proposition~\ref{PROP3}, \begin{equation*}
\lim_{\epsilon\rightarrow 0}\epsilon \log\mathbb{P}_x(Y^{\epsilon}_{T}\geqslant b)=-\frac{[b-(\frac{\alpha}{\gamma}+(x-\frac{\alpha}{\gamma})e^{-\gamma T})]^{2}}{[1-e^{-2\gamma T}] (\sigma^{2}/\gamma)}.
\end{equation*}
Moreover, the minimal cost path is as given in Proposition~\ref{PROP1} (with $a$ replaced by $b$).
\end{theorem}
\begin{proof}
Since $b(0)=\alpha >0$, by Proposition~\ref{PROP5}, $Y^{\epsilon}$ satisfies the sample path {\sc ldp} in $C_{[0,T]}([0, \infty))$ 
with the rate function
\[
I_{x}^{+}(h)=
\frac{1}{2\sigma^{2}}\int_{0}^{T}\left(h'_{t}-\alpha+\gamma h_{t})\right)^{2}\mathrm{d}t
\]
if $h \in H_x^+$ and $\infty$ else. By an argument that is similar to the one used in 
the proof of Proposition~\ref{PROP2}, $S^{+}$ is an $I^{+}_{x}$-continuity set. Then,
\[
\lim_{\epsilon\rightarrow 0}\epsilon \log\mathbb{P}_x(Y^{\epsilon}_{T}\geqslant b)=\lim_{\epsilon\rightarrow 0}\epsilon \log\mathbb{P}_x(Y^{\epsilon}_{\centerdot}\in S^{+})=-\inf_{h \in S^{+}} I^{+}_{x}(h) =-\inf_{a\geqslant b}\inf_{h \in S^{+}_{a}}I^{+}_{x}(h).
\]
We have 
\begin{eqnarray*}
\inf_{h \in S^{+}_{a}}I^{+}_{x}(h)
&=&\inf \left\{\frac{1}{2\sigma^{2}}\int_{0}^{T}(h'_{t}+\gamma h_{t}-\alpha)^{2}\mathrm{d} t, h\in H_x\cap S^{+}_{a}\right\}\\
&\geqslant& \inf \left\{\frac{1}{2\sigma^{2}}\int_{0}^{T}(h'_{t}+\gamma h_{t}-\alpha)^{2}\mathrm{d} t, h\in H_x\cap S_{a}\right\}\\
&=& \inf_{f \in S_{a}}I_{x}(f).
\end{eqnarray*}
The optimizer $f\s$ of $\inf_{f \in S_{a}}I_{x}(f)$ is always positive, for any starting point $x \geqslant 0$, due to Proposition~\ref{PROP1}. That is, $f\s \in S^{+}_{a}$. 
Conclude that $\inf_{h \in S^{+}_{a}}I^{+}_{x}(h)= \inf_{f \in S_{a}}I_{x}(f).$ 
Then the results follows immediately from Proposition~\ref{PROP1} and $a\geqslant b> x(T)$,
\begin{eqnarray*}
\lim_{\epsilon\rightarrow 0}\epsilon \log\mathbb{P}_x(Y^{\epsilon}_{T}\geqslant b)
&=& -\inf_{a\geqslant b}\frac{[a-(\frac{\alpha}{\gamma}+(x-\frac{\alpha}{\gamma})e^{-\gamma T})]^{2}}{[1-e^{-2\gamma T}] (\sigma^{2}/\gamma)}\\
&=& -\frac{[b-(\frac{\alpha}{\gamma}+(x-\frac{\alpha}{\gamma})e^{-\gamma T})]^{2}}{[1-e^{-2\gamma T}] (\sigma^{2}/\gamma)}.
\end{eqnarray*}
This proves the claim. 
\end{proof}

\section{Transient asymptotics for doubly reflected Ornstein-Uhlenbeck}\label{SEC4}
This section  computes the decay rate (\ref{DRA}), but now for $\DROU$. The case of $\DROU$ corresponds to choose the set $D=[0, d]$ 
in Theorem~\ref{LDPRD}. We can still derive an explicit expression for the optimal $\omega^\star_t$, and hence have the following simplified rate function.
\begin{proposition}\label{PROP9} Given $D=[0, d]$, the rate function $I(h)$ in Theorem~\ref{LDPRD} can be rewritten as follows:
\[
I^{++}(h)=
\frac{1}{2\sigma^{2}}\int_{0}^{T}\left(h'_{t}-b(h_{t})-\one_{\{0\}}(h_{t})b(0)^{-}+\one_{\{d\}}(h_{t})b(d)^{+}
\right)^{2}\mathrm{d}t
\]
if  $h \in H_x^{++}:=\{f\in H_x: 0\leq f\leq d\}$ and $\infty$ else.
\end{proposition}
\begin{proof}
Since $\partial D = \{0, d\}$, the rate function becomes
\begin{equation*}
I^{++}(h)=
\inf_{\omega\geqslant 0}\frac{1}{2\sigma^{2}}\int_{0}^{T}\left(h'_{t}- b(h_{t})-[\one_{\{0\}}(h_{t})-\one_{\{d\}}(h_{t})]\omega_{t}\right)^{2}\mathrm{d}t.
\end{equation*}
Let $\omega^{\star}$ denote the optimizer of the above problem. We discuss the value of $\omega^{\star}_t$ in two cases.
\begin{itemize}\item
\emph{Case 1}: $h'_{t}-b(h_{t})<0$.
If $h_{t}\in (0,d)$, then $\omega\s_{t}$ can be any value; if $h_{t}=0$, then $\omega\s_{t}=0$; if $h_{t}=d$, then $\omega\s_{t}=-(h'_{t}-b(h_{t}))$.
\item \emph{Case 2}: $h'_{t}-b(h_{t})\geqslant 0$.
If $h_{t}\in (0, d)$, then $\omega\s_{t}$ can be any value; if $h_{t}= d$, then $\omega\s_{t}=0$; if $h_{t}=0$, then $\omega\s_{t}=h'_{t}-b(h_{t})$.
\end{itemize}
As a consequence we have the following explicit expression:
\begin{equation*}
I^{++}(h)=
\frac{1}{2\sigma^{2}}\int_{0}^{T}\left(h'_{t}-b(h_{t})-\one_{\{0\}}(h_{t})(h'_{t}- b(h_{t}))^{+}+\one_{\{d\}}(h_{t})(h'_{t}- b(h_{t}))^{-}\right)^{2}\mathrm{d}t
\end{equation*}
if $h \in H_x^{++}$ and $\infty$ else. Also, $h \in [0, d]$ and $h$ is differentiable a.e. imply that $\forall t\in (0, T)$, $h'_{t}=0$ 
when $h_{t}=0$ or $h_{t}=d$. So the above expression can be further simplified to
\begin{equation*}
I^{++}(h)=
\frac{1}{2\sigma^{2}}\int_{0}^{T}\left(h'_{t}-b(h_{t})-\one_{\{0\}}(h_{t})(- b(0))^{+}+\one_{\{d\}}(h_{t})(- b(d))^{-}\right)^{2}\mathrm{d}t
\end{equation*}
if $h \in H_x^{++}$ and $\infty$ else.  
\end{proof}
\vb

For $\DROU$, $b(\cdot)$ is bounded, and as a consequence it fulfills all requirements in 
Theorem~\ref{LDPRD} and its rate function can be obtained by Proposition \ref{PROP9} directly.
Now, by a similar argument as employed in the last section, we prove that the decay rates (\ref{DRA}) for $\OU$ and $\DROU$ (and the corresponding most likely paths) coincide. 
Recall from the introduction that we have assumed ${\alpha}/{\gamma}<d$ for $\DROU$  throughout this paper. Under this assumption, the zeroth-order approximation  $x(t)$ belongs to $(0, b)$ when the starting point $x$ is in $[0, d]$. 

We consider crossing levels $d\geqslant a \geqslant b> x(T)$. Define \[S^{++}:=\{f \in C_{[0,T]}([0, d]): f(0)=x, f(T)\geqslant b\},\] so that our rare event corresponds to $Z^{\epsilon}_{\centerdot}\in S^{++}$; 
the set $S_a^{++}$ is defined as $\{f \in C_{[0,T]}([0, d]):f(0)=x, f(T)=a \}$. Finally, we arrive at the main result for $\DROU$.

\begin{theorem} Let $d\geqslant b> x(T)$. Then\begin{equation*}
\lim_{\epsilon\rightarrow 0}\epsilon \log\mathbb{P}_x(Z^{\epsilon}_{T}\geqslant b)=-\frac{[b-(\frac{\alpha}{\gamma}+(x-\frac{\alpha}{\gamma})e^{-\gamma T})]^{2}}{[1-e^{-2\gamma T}] \sigma^{2}/\gamma}.
\end{equation*}
Moreover, the minimal cost path is the one given in Proposition~\ref{PROP3}.
\end{theorem}
\begin{proof}
Since $b(0)=\alpha >0$ and $b(d)=\alpha-\gamma d< 0$, by Proposition~\ref{PROP9}, $Z^{\epsilon}$ satisfies the sample-path {\sc ldp} in 
$C_{[0,T]}([0, d])$  with the rate function
\begin{equation*}
I_{x}^{++}(h)=
\frac{1}{2\sigma^{2}}\int_{0}^{T}\left(h'_{t}-\alpha+\gamma h_{t}\right)^{2}\mathrm{d}t
\end{equation*}
if $h\in H_x^{++}$ and $\infty$ else. Due to Proposition~\ref{PROP1}, it holds that $f\s=\arg \inf_{f \in S_{a}}I_{x}(f)$ is in $[0, d]$ on $t \in [0, T]$,
when the starting point $x\in [0, d]$, $a\in [0, d]$ and ${\alpha}/{\gamma}<d$, i.e., $f\s \in S_a^{++}$.
As an immediate consequence, $f\s = \arg \inf_{\phi\in S_a^{++}} I^{++}_x(\phi)$, and $\inf_{h \in S^{++}_{a}}I^{++}_{x}(h)= \inf_{f \in S_{a}}I_{x}(f).$ The rest of proof is similar to that of 
Theorem~\ref{PROP6}. \end{proof}
\vb

As a side remark, we mention that if $d < {\alpha}/{\gamma}$, then the rate function reads
\begin{equation*}
I_{x}^{++}(h)=
\frac{1}{2\sigma^{2}}\int_{0}^{T}\left(h'_{t}-\alpha+\gamma h_{t}+\one_{\{d\}}(h_{t})(\alpha-\gamma d)\right)^{2}\mathrm{d}t.
\end{equation*}
In this case, the upper boundary does affect the transient behavior of $\DROU$; this is in line with the intuition that in this case the process is `pushed' towards the upper boundary. It is noted, however, that the functional within the integral is not differentiable with respect to $h_t$,
which seems to make analytically solving the associated variational problem challenging.

\section{Central limit theorem of loss process}\label{SEC5}
The main objective of this section is to derive a central limit theorem for the loss process $U_t$, for $t$ large. We do so relying on martingale techniques. A similar procedure can be followed for the idleness process $L_t.$

\vb

Let $h$ be a twice continuously differentiable function on ${\mathbb R}$, and $Z$ be the $\DROU$ process defined earlier. By It\^{o}'s formula, we have:
\[
\DD h(Z_{t})=\big((\alpha-\gamma Z_{t})h'(Z_{t})+\frac{\sigma^2}{2}h''(Z_{t})\big)\DD t+\sigma h'(Z_{t})\DD B_{t}+h'(Z_{t})\DD L_{t}-h'(Z_{t})\DD U_{t}.
\]
Based on the key properties of  $L$ and $U$, this reduces to
\begin{equation}\label{HH}
\DD h(Z_{t})=(\mathcal{L}h)(Z_{t})\mathrm{d}t+\sigma h'(Z_{t})\DD B_{t}+h'(0)\DD L_{t}-h'(d)\DD U_{t},
\end{equation}
where the operator ${\mathcal L}$ is defined through
\begin{equation*}
\mathcal{L}:=(\alpha-\gamma x)\frac{\DD}{\DD x}+\frac{\sigma^2}{2}\frac{\DD^2}{\DD x^2}.
\end{equation*}
We first solve the following {\sc ode} with mixed conditions.
\begin{lemma} \label{lemma7}
The {\sc ode} with real variable right hand side $q\in\rr$
\begin{equation*}
(\mathcal{L}h)(x)=q,\:\:\:\:\:0\leqslant x \leqslant d,
\end{equation*}
such that $h(0)=0,$ $h'(0)=0,$ and $h'(d)=1$,
has the unique solution
\[
q  =q_U:=\frac{\sigma^{2}}{2}
\frac{W(d)}{\int_0^d W(v)\DD v}, \:\:\:\:\:
h(x) =\frac{2q_U}{\sigma^{2}}\int_{0}^{x}\int_0^u
\frac{W(v)}{W(u)}\DD v\,\DD u,
\]
where
\begin{equation*}
W(v):=\exp\left(\frac{2\alpha v}{\sigma^2}-\frac{\gamma v^2}{\sigma^2}\right).
\end{equation*}
\end{lemma}

\begin{proof} By applying reduction of order, the {\sc ode} can be written a system of first-order ordinary differential equations:
\begin{equation*}
h'(x)=f(x),\:\:\:\:\:\:
f'(x)+\frac{2\alpha-2\gamma x}{\sigma^2}f(x)=\frac{2q}{\sigma^2}.
\end{equation*}
The integrating factor of the second first-order {\sc ode} is $W(x)$.
Hence,
\begin{equation*}
f(x)=\frac{C_{1}}{W(x)}+\frac{2q}{\sigma^2}\int_{0}^{x}\frac{W(u)}{W(x)}\DD u.
\end{equation*}
Then the general solution is
\[h(x) =C_{2}+\int_{0}^{x}f(u)\DD u=C_{2}+C_{1}\int_{0}^{x}\frac{1}{W(u)}\DD u+
\frac{2q}{\sigma^{2}}\int_{0}^{x}\int_0^u
\frac{W(v)}{W(u)}\DD v\,\DD u.
\]
Then the initial conditions $h(0)=0, h'(0)=0$ uniquely determine the values of $C_{1}, C_{2}$, while $h'(d)=1$ uniquely determines $q_U$. Hence, we obtained the desired unique solution. \end{proof}

\vb

\begin{proposition} The loss process $U$ satisfies the central limit theorem, with $\eta_U^2$ defined in (\ref{ergo}),
\begin{equation*}
\frac{U_{t}-q_Ut}{\sqrt{t}}\Rightarrow \mathcal{N}(0,\eta_U^{2}), \:\text{as  } t\rightarrow \infty.
\end{equation*}
\end{proposition}
\begin{proof} We insert the unique solution $h(x)$ of Lemma~\ref{lemma7} into (\ref{HH}). Since $h'(0)=0, h'(d)=1$, and $(\mathcal{L}h)(Z_{t})=q_U$, we have the following integral expression:
\begin{equation*}
U_{t}-q_Ut+h(Z_{t})-h(Z_{0})=\sigma \int_{0}^{t}h'(Z_{s})\mathrm{d}B_{s}.
\end{equation*}
We then observe that $M_{t}:=U_{t}-q_Ut+h(Z_{t})-h(Z_{0})$ is a zero-mean square integrable martingale. As usual, $\langle M\rangle$ denotes the quadratic variation process of $M$. By the ergodic theorem \cite[p.\ 134]{MR0346904},
\begin{equation} \label{ergo}
t^{-1}\langle M\rangle_{t}=t^{-1}\sigma^{2} \int_{0}^{t}h'(Z_{s})^{2}\DD s \stackrel{\mathbb{P}}{\rightarrow} \sigma^{2} \int_{0}^{d}h'(x)^{2}\pi(\DD x)=:\eta_U^{2},
\end{equation}
where $\pi$ is the stationary distribution corresponding $Z_{t}$. The density of $\pi$ is obtained in \cite[Prop.\ 1]{MR1993278} as, with $N(m,s^2)$ denoting a Normal random variable with mean $m$ and variance $s^2$,
\begin{eqnarray*}
\pi(x)
&=& \frac{\DD}{\DD x}{\mathbb P}\left(\left.N\left(\frac{\alpha}{\gamma}, \frac{\sigma^{2}}{2\gamma}\right)\leq x\,\right|\,0 \leqslant N\left(\frac{\alpha}{\gamma}, \frac{\sigma^{2}}{2\gamma}\right) \leqslant d\right)\\&=&\sqrt{\frac{2\gamma}{\sigma^{2}}}\frac{\varphi\left((x-\frac{\alpha}{\gamma})\sqrt{\frac{2\gamma}{\sigma^{2}}}\right)}{\Phi\left((d-\frac{\alpha}{\gamma})\sqrt{\frac{2\gamma}{\sigma^{2}}}\right)-\Phi\left((-\frac{\alpha}{\gamma})\sqrt{\frac{2\gamma}{\sigma^{2}}}\right)},\end{eqnarray*}
where $\varphi$ and $\Phi$ are the density function and cumulative density function of a standard Normal random variable.
Then,
\[
\eta^{2}_U= \frac{4q_U^{2}}{\sigma^{2}}\int_{0}^{d}
\left(\int_0^x\frac{W(v)}{W(x)}\DD v\right)^2 \pi(x)\,\DD x.\]
Then by the martingale central limit theorem \cite[Thm.\ 2.1]{MR2952852},
$
t^{-\frac{1}{2}}M_{t}\Rightarrow \mathcal{N}(0,\eta_U^{2})$ as  $t\rightarrow \infty.$
Since $Z \in [0, d]$ and $h$ is continuous, $h(Z)$ is bounded. So, \[
\frac{h(Z_{t})-h(Z_{0})}{\sqrt{t}}\rightarrow 0\] a.s.\ as $t\rightarrow \infty$, which implies the claim.
\end{proof}

\vb

The loss process at $0$ can be treated analogously. Define
\[q_L:=\frac{\sigma^2}{2}\frac{1}{\int_0^d W(v)\DD v},\:\:\:\:
\eta_L^2:=\sigma^2\int_0^d \left(-\frac{1}{W(x)}+\frac{2q_L}{\sigma^2}\int_0^x\frac{W(v)}{W(x)}\DD x\right)^2\pi(x)\,{\rm d}x.\]

\begin{proposition} The loss process $L$ satisfies the central limit theorem
\begin{equation*}
\frac{L_{t}-q_Lt}{\sqrt{t}}\Rightarrow \mathcal{N}(0,\eta_L^{2}), \:\text{as  } t\rightarrow \infty.
\end{equation*}
\end{proposition}

\section{Discussion on large deviations of loss process}\label{SEC6}
Realizing that $q_U$ is the mean rate at which $U$ increases, it is seen that for $c> q_U$, the event $\{U_{t}> ct\}$ for $t$ large is rare. We analyze the decay rate of its large deviations probability  $\lim_{t \rightarrow \infty} t^{-1}\log \mathbb{P}_{z}(U_{t}>ct)$ by appealing to the G\"{a}rtner-Ellis theorem. Here and below we use the notation $\mathbb{P}_{z}(\cdot)$ for the probabilities of events in terms of $Z$ if $Z$ starts in $z$, and $\mathbb{E}_{z}(\cdot)$ the corresponding expectation.

To apply this result,
our first task is to prove that $t^{-1}\log \mathbb{E}_{z}\exp(\theta U_t)$ converges, as $t\to\infty$, to a function $\psi(\theta)$ (pointwise in $\theta \in {\mathbb R}$; this function is often referred to as the limiting cumulant function).  We apply It\^{o}'s formula to obtain a martingale which contains $\exp(\theta U_{t})$. Then we set $h'(0)=0$, $h'(d)=\theta$. Similarly to what we did above, we obtain
\begin{equation*}
\theta U_{t}-\int_{0}^{t}(\mathcal{L}h)(Z_{s})\mathrm{d}s+h(Z_{t})-h(Z_{0})=\sigma \int_{0}^{t}h'(Z_{s})\mathrm{d}B_{s}.
\end{equation*}
The quadratic variation process of $\sigma \int_{0}^{\cdot}h'(Z_{s})\mathrm{d}B_{s}$ is $\sigma^{2}\int_{0}^{\cdot}h'(Z_{s})^{2}\mathrm{d}s$. Since $Z \in [0, d]$ and $h$ is twice continuously differentiable, Novikov's condition is satisfied. We have the following exponential martingale:
\begin{equation*}
M_t=\exp\left(\sigma \int_{0}^{t}h'(Z_{s})\mathrm{d}B_{s}-\frac{\sigma^{2}}{2}\int_{0}^{t}h'(Z_{s})^{2}\mathrm{d}s\right).
\end{equation*}
Also,
\begin{equation*}
M_t=\exp\left(\theta U_{t}-\int_{0}^{t}[(\mathcal{L}h)(Z_{s})+\frac{\sigma^{2}}{2}h'(Z_{s})^{2}]\mathrm{d}s+h(Z_{t})-h(Z_{0})\right).
\end{equation*}
Since $h(Z)$ is bounded and $\mathbb{E}_{z}M_t=1$ for all $t\geqslant 0$,
\begin{equation}\label{UNI}
\frac{1}{t} \log \mathbb{E}_{z}\exp\left(\theta U_{t}-\int_{0}^{t}\left((\mathcal{L}h)(Z_{s})+\frac{\sigma^{2}}{2}h'(Z_{s})^{2}\right)\mathrm{d}s\right)\rightarrow 0, \:\:\text{as  } t\rightarrow \infty.
\end{equation}
Suppose the integrand in (\ref{UNI}) is merely a function of $\theta$, i.e.
$(\mathcal{L}h)(Z_{s})+({\sigma^{2}}/{2})\,h'(Z_{s})^{2}=\psi(\theta)$, then we will have the desired limit, i.e.,
\begin{equation*}
\lim_{t \rightarrow \infty}\frac{1}{t}\log \mathbb{E}_{z}\exp(\theta U_{t})=\psi(\theta).
\end{equation*}
We propose to prove this as follows. Firstly, by assuming the existence of $\psi(\theta)$, we are to solve the second-order nonlinear {\sc ode}
\begin{equation}\label{ODE}
(\mathcal{L}h)(x)+\frac{\sigma^{2}}{2}h'(x)^{2}=\psi(\theta),\:\:\:\:0\leqslant x \leqslant d,
\end{equation}
such that $h(0)=0,$ $h'(0)=0.$ Then $ h'(d)$ is a function of $\psi(\theta)$, so we can write it as $h'(d, \psi(\theta)).$ Secondly, we try to prove that there exists a unique root $\psi(\theta)$ of equation $h'(d, \psi(\theta))=\theta$ for any $\theta \in \mathbb{R}.$ It is readily seen that (\ref{ODE}) with the given initial conditions has a unique solution. Letting $h(x)$ be this solution, we can apply change of variables $y(x)=\exp(h(x))$, so as  to obtain
\begin{equation*}
\sigma^{2}y''(x)+(-2\gamma x+2\alpha)y'(x)-2\psi(\theta)y(x)=0, \:\:\:\: y(0)=1,\:\:\:\: y'(0)=0.
\end{equation*}
We can solve it explicitly, but we cannot verify the existence of a function $\psi(\theta)$ such that $h'(d, \psi(\theta))=\theta$. This problem is caused by fact that  the solution is rather involved (expressed in terms of  Kummer's series \cite{MR2001201}). 

In the context of reflected Brownian motion
Zhang and Glynn \cite{MR2771195} managed to solve the problem of identifying the limiting cumulant function $\psi(\theta)$  using the approach
followed above. This is an example which shows that it is not always possible to apply methods that work in the case of constant drift in cases with
state-dependent drift (such as $\DROU$); cf.\  \cite{MR1993278}.

{\small
\bibliographystyle{plain}

}

\end{document}